\documentclass[reqno,12pt]{amsart}
\usepackage{epsfig,amscd,amssymb,amsmath,amsfonts}

\usepackage{amsmath}
\usepackage{graphicx}
\usepackage{lscape}
\usepackage{amsthm,color}
\usepackage{tikz}
\usepackage{multirow}
\usepackage{booktabs}
\usepackage{array}
\usepackage{tabularx}
\usepackage{adjustbox}
\usepackage{diagbox}
\usepackage{makecell}
\usetikzlibrary{graphs}
\usetikzlibrary{graphs,quotes}
\usetikzlibrary{decorations.pathmorphing}

\tikzset{snake it/.style={decorate, decoration=snake}}
\tikzset{snake it/.style={decorate, decoration=snake}}

\usetikzlibrary{decorations.pathreplacing,decorations.markings,snakes}

\newtheorem{theorem}{Theorem}[section]
\newtheorem{lemma}[theorem]{Lemma}
\newtheorem{proposition}[theorem]{Proposition}
\theoremstyle{definition}
\newtheorem{definition}[theorem]{Definition}
\newtheorem{corollary}[theorem]{Corollary}

\newtheorem{example}[theorem]{Example}

\theoremstyle{remark}
\newtheorem{remark}[theorem]{Remark}

\numberwithin{equation}{section}

\usepackage[margin=1.05in]{geometry}
\usepackage[colorlinks]{hyperref}
\usepackage{siunitx}
\makeatletter
\let\@wraptoccontribs\wraptoccontribs
\makeatother
\begin{document}
 \title{Quadratic Perturbations of Markov Systems}

\author{}

\author{Riddhi Pratim Ghosh}
\address{Department of Mathematics and Statistics, Bowling Green State University, Bowling Green, OH 43403, United States.}

\email{rpghosh@bgsu.edu}

\author{Mihai D. Staic}
\address{Department of Mathematics and Statistics, Bowling Green State University, Bowling Green, OH 43403, United States. }
\address{Institute of Mathematics of the Romanian Academy, PO.BOX 1-764, RO-70700 Bu\-cha\-rest, Romania.}

\email{mstaic@bgsu.edu}

\author{Sofia Stancu}
\address{ Georgia Institute of Technology, Atlanta, GA 30332, USA }

\email{sstancu3@gatech.edu}


 \begin{abstract}
In this paper we study a quadratic dynamical system  that is a perturbation of a system of Markov processes. We give necessary and sufficient conditions for such a perturbation to be stochastic, show the existence of fixed points, and present a few examples and possible applications. Finally, we derive explicit contraction conditions that guarantee uniqueness of the fixed point and geometric convergence to it from every initial state. 
 \end{abstract}

 \maketitle

\section{Introduction}

Markov processes are some of the simplest yet surprisingly efficient models for real-world phenomena. Essentially,  a Markov process consists of a sequence of random variables $Y(1),Y(2),\dots, Y(m), \dots$ that satisfy the condition  
$$P(Y(m+1)=a\; |\; Y(m)=a_m,\dots, Y(1)=a_1)=P(Y(m+1)=a\;|\; Y(m)=a_m).$$
In other words, they require that the probability of a future event depends only on the current state. Applications of Markov processes include models for population growth and decline, animal migration, diffusion of particles, random walks, risk assessment,  web-page search and ranking, etc.  For general results and applications of Markov processes one can consult \cite{dia,k,n,wdj}.

In this paper we introduce a dynamical system that is a quadratic perturbation of a system of Markov processes. More precisely, we start with a finite number of  discrete Markov processes and modify them by adding perturbation factors proportional to the inner product of the corresponding probability vectors. For example, consider two discrete random variables $Y_1$ and $Y_2$ in $\mathbb{R}^2$ that behave as Markov processes governed by the stochastic matrices $A_1$ and $A_2$, respectively
\begin{eqnarray}
\begin{cases}
Y_1(m+1)=A_1Y_1(m)\\
Y_2(m+1)=A_2Y_2(m).
\end{cases}
\label{seq1}
\end{eqnarray}
To get a perturbation of this system of Markov processes we take $b_{1,2}$ and  $b_{2,1}\in \mathbb{R}^2$ to be two null-sum vectors, and define  a dynamical system determined by 
\begin{eqnarray}
\begin{cases}
X_1(m+1)=A_1X_1(m)+\langle X_1(m),X_2(m)\rangle b_{1,2}\\
X_2(m+1)=A_2X_2(m)+\langle X_2(m),X_1(m)\rangle b_{2,1}.\\
\end{cases}
\label{seq2}
\end{eqnarray}
In general the above equations do not define a stochastic system (see Example \ref{example1}). The issue is that the entries of the vectors $X_1(m+1)$ and $X_2(m+1)$ need not be positive. In this paper we study this type of model and give necessary and sufficient conditions  for the system to be stochastic under arbitrary sign reversal of the vectors $b_{i,j}$. We also discuss the existence and uniqueness of fixed points and give a few examples.

Studying systems of interacting random variables is not a new idea. Similar problems arise in interacting populations and interacting particle systems (see, for example, \cite{a,c,fh,li}). There is also an extensive literature on quadratic stochastic operators and nonlinear Markov operators; see, for example, \cite{qso,gj, sab}. These works study nonlinear maps of probability simplices, their fixed points, trajectories, ergodicity, and contraction properties. Our model has a different structure: each component has a prescribed linear Markov evolution, while interactions enter through the terms
\[
\langle X_i(m),X_j(m)\rangle b_{i,j}.
\]
Thus, the baseline Markov dynamics are separated from the quadratic interactions, and preservation of the product of probability simplices becomes an explicit admissibility problem. In particular, we characterize stochasticity under arbitrary sign reversals of the perturbation vectors and relate strict admissibility to global contraction and convergence.

The paper is organized as follows. In Section \ref{Section2} we recall a few basic results and notations that are used throughout  the paper. In Section \ref{Section3} we discuss the two-dimensional case  and give necessary and sufficient conditions for Equation (\ref{seq2}) to define a stochastic process under arbitrary sign reversal of the vectors $b_{i,j}$. We show that this stochastic system has a fixed point, present an algorithm for computing the fixed points and give examples. 

In Section \ref{Section4} we deal with the general case of $n$ discrete random variables $X_1,\dots X_n$ in $\mathbb{R}^d$. We introduce a nonlinear operator $T({\bf A};{\bf b})$ and give necessary and sufficient conditions for the corresponding system to define a stochastic process under arbitrary sign reversal of ${\bf b}$. We show the existence of a fixed point and present a few examples. In Section \ref{Section5} we discuss the uniqueness of the fixed point giving explicit conditions that depend on the parameters ${\bf A}$ and ${\bf b}$. In Section \ref{Section6} we discuss interpretations of the model, possible extensions, and directions for future research.

\section{Preliminaries}
\label{Section2}

For completeness we recall a few notations and results that will be  used throughout this paper. We denote by  $\mathbb{R}$ the field of real numbers; for vectors we use the column notation. 
\begin{definition} A {\it probability vector} is a vector $v=(v^{(i)})_{1\leq i\leq d}\in \mathbb{R}^d$ such that $v^{(i)}\geq 0$  for all $1\leq i\leq d$ and  ${\displaystyle \sum_{s=1}^dv^{(s)}=1}$. \\
A {\it null-sum vector} is a vector $v=(v^{(i)})_{1\leq i\leq d}\in \mathbb{R}^d$  such that  ${\displaystyle \sum_{s=1}^dv^{(s)}=0}$.\\
A {\it stochastic matrix} is a matrix $A=(a_{i,j})_{1\leq i,j\leq d}\in M_d(\mathbb{R})$ such that $a_{i,j}\geq 0$ for all $1\leq i, j\leq d$, and ${\displaystyle \sum_{s=1}^da_{s,j}=1}$  for all $1\leq j\leq d$. \\
The inner product of two vectors  $v_1=(v_1^{(i)})_{1\leq i\leq d}$ and $v_2=(v_2^{(i)})_{1\leq i\leq d}\in \mathbb{R}^d$  is defined as 
$$\langle v_1, v_2 \rangle={\displaystyle \sum_{s=1}^dv^{(s)}_1v^{(s)}_2}.$$
\end{definition} 

\begin{remark} \label{remark0}
It is well known that if $A$ is a stochastic matrix and $v$ is a probability vector then $Av$ is a probability vector. 
\end{remark}

\begin{definition} For each $d\in \mathbb{N}^*$ we define $$\Delta_{d-1}=\left\{\begin{pmatrix}
x_1\\
\vdots\\
x_d
\end{pmatrix}\in \mathbb{R}^d~|~x_i\geq 0, {\displaystyle \sum_{i=1}^dx_i=1}\right\}.$$ 
\end{definition}
Note that $\Delta_{d-1}$ is nothing else but the set of all probability vectors in $\mathbb{R}^d$. It is known that $\Delta_{d-1}$ is a compact and convex subset in $\mathbb{R}^d$. In particular $(\Delta_{d-1})^n$ is a compact and convex subset of $(\mathbb{R}^d)^n$. 

The following remark is trivial but it will be used later in the paper. So, for completeness we give a short proof. 
\begin{remark} \label{remark1} If $v_1$, $v_2\in \Delta_{d-1}$ then $0\leq \langle v_1,v_2\rangle \leq 1$. Indeed since $v_i^{(j)}$ are positive we obviously have that the inner product is positive. Moreover, since $0\leq v_i^{(j)}\leq 1$ we have 
\begin{eqnarray*}
\langle v_1,v_2\rangle&=& \left\|v_1\right\|\left\|v_2\right\|cos(\theta) \leq \left\|v_1\right\|\left\|v_2\right\|\\
&=&{\displaystyle \sqrt{\sum_{s=1}^d(v_1^{(s)})^2}\sqrt{\sum_{s=1}^d(v_2^{(s)})^2}}\leq {\displaystyle \sqrt{\sum_{s=1}^dv_1^{(s)}}\sqrt{\sum_{s=1}^dv_2^{(s)}}}= 1.
\end{eqnarray*}
\end{remark}

Let $(X,d)$ be a metric space. We say that a map $T:X\to X$ is a contraction if there exists $0<q<1$ such that $d(T(x),T(y))\leq q d(x,y)$ for all $x,y\in X$. We say that $x$ is a fixed point for $T$ if $T(x)=x$.  With this notation, recall Banach's fixed-point theorem.

\begin{theorem} \cite{rb} Let $(X,d)$ be a nonempty complete metric space and $T:X\to X$ a contraction. Then $T$ has a unique fixed point in $X$.
\end{theorem}
Finally, we recall Brouwer's fixed-point theorem. 
\begin{theorem} \cite{r} Let $K$ be a nonempty compact and convex subset of $\mathbb{R}^n$. Then every continuous function $f:K\to K$ has a fixed point. \label{FixedPoint}
\end{theorem}

\section{A Basic Example}

\label{Section3}

As a warm-up for the general construction, in this section we discuss the case of two Markov processes in $\mathbb{R}^2$. The main goal is to build intuition and provide an explicit computation. 

As mentioned in the introduction,  in general Equation (\ref{seq2}) does not define a stochastic system. Let's see such a situation. 
\begin{example} \label{example1} Consider the system determined by Equation (\ref{seq2}). Take the stochastic matrices $A_1=\begin{pmatrix}
0.8 & 0.4\\
0.2 & 0.6 
\end{pmatrix}$,  $A_2=\begin{pmatrix} 
0.5 & 0.2\\
0.5 & 0.8 
\end{pmatrix}$ and  the null-sum vectors $b_{1,2}=-b_{2,1}=\begin{pmatrix}
0.3\\
-0.3
\end{pmatrix}$. If $X_1(0)=X_2(0)=\begin{pmatrix}
1\\
0
\end{pmatrix}$ then one  gets 
$$X_1(1)=A_1X_1(0)+\langle X_1(0),X_2(0)\rangle b_{1,2}=\begin{pmatrix}
1.1\\
-0.1
\end{pmatrix},$$ which is not a probability vector. 
\end{example}

Next we give necessary and sufficient conditions for Equation (\ref{seq2}) to define a stochastic system under arbitrary sign reversals of the null-sum vectors $b_{i,j}$. 

\begin{proposition}  \label{prop1} Let $A_1=	\begin{pmatrix}
a_1 & b_1\\
c_1 & d_1 
\end{pmatrix}$ and $A_2=	\begin{pmatrix}
a_2 & b_2\\
c_2 & d_2 
\end{pmatrix}$ be $2\times 2$ stochastic matrices, and $b_{1,2}=\begin{pmatrix}
\alpha_{1,2}^{(1)}\\
\alpha_{1,2}^{(2)}
\end{pmatrix}\in \mathbb{R}^2,$ $b_{2,1}=\begin{pmatrix}
\alpha_{2,1}^{(1)}\\
\alpha_{2,1}^{(2)}
\end{pmatrix}\in \mathbb{R}^2$ null-sum vectors. 
We define the nonlinear operator 

$$T(A_1,A_2;b_{1,2}, b_{2,1}):\mathbb{R}^2\times \mathbb{R}^2\to \mathbb{R}^2\times \mathbb{R}^2,$$ 
determined by 
\begin{equation}\label{eq1}
T(A_1,A_2;b_{1,2},b_{2,1})\begin{bmatrix}
v_1\\
v_2
\end{bmatrix}=\begin{bmatrix}
A_1v_1+\langle v_1,v_2\rangle b_{1,2}\\
A_2v_2+\langle v_2,v_1\rangle b_{2,1}
\end{bmatrix}.
\end{equation}
The following are equivalent. 
\begin{enumerate}
\item For each of the four possible choices of $\pm$ sign of the null-sum vectors $b_{1,2}$ and $b_{2,1}$,  the maps $T(A_1,A_2;\pm b_{1,2}, \pm b_{2,1})$ induce  well-defined nonlinear operators from $\Delta_1\times \Delta_1$ to $\Delta_1\times \Delta_1$. 
\item The entries of  $A_1$, $A_2$, $b_{1,2}$ and $b_{2,1}$ satisfy the  following inequalities
\begin{eqnarray}
\begin{aligned}
|\alpha_{1,2}^{(1)}|\leq \min \{a_1,b_1\},\\
|\alpha_{2,1}^{(1)}|\leq \min \{a_2,b_2\},\\
|\alpha_{1,2}^{(2)}|\leq \min \{c_1,d_1\}, \\
|\alpha_{2,1}^{(2)}|\leq \min \{c_2,d_2\}.\label{cond2}
\end{aligned}
\end{eqnarray}
\end{enumerate} 

\end{proposition}
\begin{proof} First we show that condition (\ref{cond2}) is sufficient. Indeed, let  $v_i=\begin{pmatrix}
v_i^{(1)}\\
v_i^{(2)}
\end{pmatrix}\in \Delta_1$. The sum of the components of the vector $\begin{pmatrix}
w_1^{(1)}\\
w_1^{(2)}
\end{pmatrix}=A_1v_1 \pm \langle v_1,v_2\rangle b_{1,2}$ is  given by 
\begin{eqnarray*}w_1^{(1)}+w_1^{(2)}&=&a_1v_1^{(1)}+b_1v_1^{(2)}\pm \langle v_1,v_2\rangle\alpha_{1,2}^{(1)}+c_1v_1^{(1)}+d_1v_1^{(2)}\pm \langle v_1,v_2\rangle\alpha_{1,2}^{(2)}\\
&=&(a_1+c_1)v_1^{(1)}+(b_1+d_1)v_1^{(2)}\pm \langle v_1,v_2\rangle(\alpha_{1,2}^{(1)}+\alpha_{1,2}^{(2)})\\
&=&v_1^{(1)}+v_1^{(2)}=1.
\end{eqnarray*}
Here we  use that $a_1+c_1=b_1+d_1=1$ (because $A_1$ is stochastic), $\alpha_{1,2}^{(1)}+\alpha_{1,2}^{(2)}=0$ ($b_{1,2}$ is a null-sum vector), and $v_1^{(1)}+v_1^{(2)}=1$ ($v_1\in \Delta_1$).  A similar computation shows that $w_2^{(1)}+w_2^{(2)}=1$. 

We still need to show that the entries of these vectors are positive. Indeed, for example 
\begin{eqnarray*} w_1^{(1)}&=&a_1v_1^{(1)}+b_1v_1^{(2)}+\langle v_1,v_2\rangle\alpha_{1,2}^{(1)}\\
&\geq& min\{a_1,b_1\}(v_1^{(1)}+v_1^{(2)})+min\{0, \langle v_1,v_2\rangle\alpha_{1,2}^{(1)}\}\\
&\geq& min\{a_1,b_1\}+min\{0, \alpha_{1,2}^{(1)}\}\geq \left|\alpha_{1,2}^{(1)}\right|+min\{0, \alpha_{1,2}^{(1)}\}\geq 0.\\
\end{eqnarray*}
Here we use  that $a_1,b_1\geq 0$ ($A_1$ is stochastic), $v_1^{(1)}, v_1^{(2)}\geq 0$, $v_1^{(1)}+v_1^{(2)}=1$ ($v_1\in \Delta_1$). Moreover, since $v_i\in \Delta_1$ we know from Remark \ref{remark1} that $0\leq \langle v_1,v_2\rangle \leq 1$. 

A similar argument shows that all the other $w_i^{(j)}$ are nonnegative, which shows that  we have a well-defined nonlinear operator $$T(A_1,A_2;\pm b_{1,2},\pm b_{2,1}): \Delta_1\times \Delta_1\to \Delta_1\times \Delta_1.$$

Conversely,  suppose that $T(A_1,A_2;\pm b_{1,2},\pm b_{2,1})$ is well-defined from $\Delta_1\times \Delta_1$ to $\Delta_1\times \Delta_1$ for each of the four possible choices of $\pm$ sign. 
Let  $\begin{pmatrix}
w_1^{(1)}\\
w_1^{(2)}
\end{pmatrix}=A_1v_1 \pm \langle v_1,v_2 \rangle b_{1,2}$. 
Taking $v_1=v_2=\begin{pmatrix}
1\\
0
\end{pmatrix}\in \Delta_1$ we get 
\begin{eqnarray*} w_1^{(1)}&=&a_1v_1^{(1)}+b_1v_1^{(2)}\pm \langle v_1,v_2\rangle \alpha_{1,2}^{(1)}\\
&=&a_1\pm \alpha_{1,2}^{(1)}\geq 0,
\end{eqnarray*}
or equivalently $-a_1\leq \alpha_{1,2}^{(1)}\leq a_1$. 

Next, if $v_1=v_2=\begin{pmatrix}
0\\
1
\end{pmatrix}\in \Delta_1$ we get 
\begin{eqnarray*} w_1^{(1)}&=&a_1v_1^{(1)}+b_1v_1^{(2)}\pm\langle v_1,v_2\rangle\alpha_{1,2}^{(1)}\\
&=&b_1\pm \alpha_{1,2}^{(1)}\geq 0,
\end{eqnarray*}
or equivalently $-b_1\leq \alpha_{1,2}^{(1)}\leq b_1$. 
Combining these two inequalities we obtain 
$$|\alpha_{1,2}^{(1)}|\leq \min \{a_1,b_1\}.$$ Similarly,  one can prove the other  inequalities. 
\end{proof}

\begin{corollary}  \label{corol1} Under the assumptions of Proposition \ref{prop1} the operator $T(A_1,A_2;b_{1,2},b_{2,1})$ has  a fixed point in $\Delta_1\times \Delta_1$. 
\end{corollary}
\begin{proof}
Since $\Delta_1^2$ is a compact and convex set, the result is a direct consequence of Theorem \ref{FixedPoint} and Proposition \ref{prop1}. 
\end{proof}


\begin{remark} With the notation from the introduction, Equation (\ref{seq2}) can be written as  
$$\begin{bmatrix}
X_1(m+1)\\
X_2(m+1)
\end{bmatrix}=T(A_1,A_2;b_{1,2}, b_{2,1})\begin{bmatrix}
X_1(m)\\
X_2(m)
\end{bmatrix}.$$
In particular, if the conditions from Proposition \ref{prop1} hold, then Equation (\ref{seq2})  defines a stochastic system that has a fixed point. 
\end{remark}

The next example shows how one can find the fixed vectors of the nonlinear operator  $T(A_1,A_2;b_{1,2},b_{2,1})$  (the existence of such vectors is known from Corollary \ref{corol1}).
\begin{example} \label{example2}

Let
\[
v_1=
\begin{pmatrix}
x\\
1-x
\end{pmatrix},
\qquad
v_2=
\begin{pmatrix}
y\\
1-y
\end{pmatrix},
\]
where $x,y\in[0,1]$, so that
\[
(v_1,v_2)\in\Delta_1\times\Delta_1.
\]
Suppose that
\[
T(A_1,A_2;b_{1,2},b_{2,1})
\begin{bmatrix}
v_1\\
v_2
\end{bmatrix}
=
\begin{bmatrix}
v_1\\
v_2
\end{bmatrix},
\]
or equivalently,

\begin{eqnarray}\label{eq4}
\begin{cases}
v_1=A_1v_1+\langle v_1,v_2\rangle \begin{pmatrix}
\alpha_{1,2}\\
-\alpha_{1,2}
\end{pmatrix}\\
v_2=A_2v_2+\langle v_2,v_1\rangle \begin{pmatrix}
\alpha_{2,1}\\
-\alpha_{2,1}
\end{pmatrix}.\\
\end{cases}
\end{eqnarray}

First, if $\alpha_{1,2}=0$ then $v_1$ is an eigenvector for $A_1$ corresponding to the eigenvalue $\lambda=1$.  Since $v_1$ is a stochastic vector, one can  determine $v_1$ and then solve the second equation for $v_2$. A similar statement is true if $\alpha_{2,1}=0$. So, without loss of generality we may assume that $\alpha_{1,2}\neq 0$ and $\alpha_{2,1}\neq 0$. 

If we multiply the above two equations by $\alpha_{2,1}$ and $\alpha_{1,2}$ respectively, and take the difference we get $$\alpha_{2,1}v_1-\alpha_{1,2}v_2=\alpha_{2,1}A_1v_1-\alpha_{1,2}A_2v_2.$$ This gives the identity 
$$\alpha_{2,1}(a_1x+b_1(1-x))-\alpha_{1,2}(a_2y+b_2(1-y))=\alpha_{2,1}x-\alpha_{1,2}y,$$
or equivalently 
$$\alpha_{2,1}(a_1-b_1-1)x-\alpha_{1,2}(a_2-b_2-1)y+\alpha_{2,1}b_1-\alpha_{1,2}b_2=0.$$

Notice that if $a_1-b_1-1=0$ then $a_1=1$ and $b_1=0$ (that is because $0\leq a_1, b_1\leq 1$). Moreover, since $A_1$ is a stochastic matrix, this implies $c_1=0$ and $d_1=1$, which means that $A_1=I_2$. Finally,  the condition from Proposition \ref{prop1} forces  $b_{1,2}=0$ (i.e. there is no perturbation). A similar statement is true when $a_2-b_2-1=0$. 

So, without loss of generality we may assume that $a_1-b_1-1\neq 0$ and $a_2-b_2-1\neq 0$. Under this assumption, one can solve for $y$ in terms of $x$: 
\begin{equation}\label{eqy}
y= \frac{\alpha_{2,1}(a_1-b_1-1)x+\alpha_{2,1}b_1-\alpha_{1,2}b_2}{\alpha_{1,2}(a_2-b_2-1)}.
\end{equation}
Finally, from Equation (\ref{eq4}) we know that 
$$a_1x+b_1(1-x)+(xy+(1-x)(1-y))\alpha_{1,2}=x.$$
Therefore, using the formula for $y$ from Equation (\ref{eqy}), we obtain a quadratic equation in $x$.  Solving this quadratic equation gives the fixed vectors for $T(A_1,A_2;b_{1,2},b_{2,1})$. 
\end{example}

Let's see an example with an explicit computation.
\begin{example} \label{example3} Since childhood, Alice and Bob have held separate birthday parties on the same day but at different times at the local library.  The library has a lounge and a patio  area. The attendees of Alice's party usually move between these two spaces according to a Markov process governed by the matrix  $A=\begin{pmatrix}
0.8 & 0.4\\
0.2 & 0.6 
\end{pmatrix}$. For Bob's party, the distribution is governed by the matrix  $B=\begin{pmatrix}
0.5 & 0.2\\
0.5 & 0.8 
\end{pmatrix}$. Finding the eigenvectors corresponding to $\lambda=1$, one can see that  towards the end of each party the distribution of guests between the lounge and patio area is given by the vectors $v_A=\begin{pmatrix}
2/3\\
1/3
\end{pmatrix}\approx \begin{pmatrix}
0.6667\\
0.3333
\end{pmatrix}$, and $v_B=\begin{pmatrix}
2/7\\
5/7
\end{pmatrix}\approx \begin{pmatrix}
0.2857\\
0.7143
\end{pmatrix}$ respectively. 

Due to an unfortunate double booking, this year Alice and Bob must share the library space for their birthdays. Let's assume that at this party the two groups behave according to the model given by Equation (\ref{seq2}), with interference given by the null-sum vectors 
$$b_{1,2}=-b_{2,1}=\begin{pmatrix}
0.1\\
-0.1
\end{pmatrix}.$$ 
Using the notation from Example \ref{example2}, we have 
$$y=\frac{6-6x}{7},$$ and the corresponding quadratic equation 
$$12x^2+31x-29=0.$$
This has two solutions $\frac{-31\pm\sqrt{2353}}{24}$. Only the positive one gives a point in $\Delta_1$, which means that $x=\frac{-31+\sqrt{2353}}{24}\approx 0.729489$, and so 
$$v_1=\begin{pmatrix}
\frac{-31+\sqrt{2353}}{24}\\
\frac{55-\sqrt{2353}}{24}
\end{pmatrix} \approx  \begin{pmatrix}
0.7295\\
0.2705
\end{pmatrix},$$ and $$v_2=
\begin{pmatrix}
\frac{55-\sqrt{2353}}{28}\\
\frac{-27+\sqrt{2353}}{28}
\end{pmatrix}
\approx
\begin{pmatrix}
0.2319\\
0.7681
\end{pmatrix}.$$ 
\end{example}

\begin{remark} One should note that since $b_{1,2}$, $b_{2,1}\in \mathbb{R}^2$ are null-sum vectors we have that $\alpha_{1,2}^{(2)}=-\alpha_{1,2}^{(1)}$ and $\alpha_{2,1}^{(2)}=-\alpha_{2,1}^{(1)}$. So, the four conditions from Equation (\ref{cond2}) are equivalent to  
\begin{eqnarray}
\begin{aligned}
|\alpha_{1,2}^{(1)}|\leq \min \{a_1,b_1,c_1,d_1\},\\
|\alpha_{2,1}^{(1)}|\leq \min \{a_2,b_2,c_2,d_2\}.\label{cond22}
\end{aligned}
\end{eqnarray}
We prefer to use the expanded condition in Proposition \ref{prop1} because that formulation aligns better with the generalization we prove in Section \ref{Section4}. 
\end{remark}

 \section{The General Case}

\label{Section4}

In this section we consider the general case of $n$ variables in $\mathbb{R}^d$.  We give conditions for our dynamical system to be stochastic, show the existence of fixed points and give a few examples.

Consider  stochastic matrices $A_1,\dots, A_n\in M_d(\mathbb{R})$ and perturbation null-sum vectors $b_{i,j}\in \mathbb{R}^d$ for all $1\leq i\neq j\leq n$ (for convenience we denote $b_{i,i}=0$). We want to find conditions such that the system 
\begin{eqnarray}
\begin{cases}
X_1(m+1)=A_1X_1(m)+{\displaystyle \sum_{s=1}^n\langle X_1(m),X_s(m)\rangle b_{1,s}}\\
X_2(m+1)=A_2X_2(m)+{\displaystyle \sum_{s=1}^n\langle X_2(m),X_s(m)\rangle b_{2,s}}\\
  \hspace{3cm} \vdots \\
X_n(m+1)=A_nX_n(m)+{\displaystyle\sum_{s=1}^n\langle X_n(m),X_s(m)\rangle b_{n,s}},\\
\end{cases}
\label{seq3}
\end{eqnarray}
defines a stochastic process under arbitrary sign reversal of the vectors $b_{i,j}$. We have the following result.

\begin{theorem}  \label{prop2} Let ${\bf A}=(A_i)_{1\leq i\leq n}$ such that   $A_i=(a_{i}^{(s,t)})_{1\leq s,t\leq d}\in M_d(\mathbb{R})$ is a stochastic $d\times d$ matrix for each $1\leq i\leq n$, and ${\bf b}=(b_{i,j})_{1\leq i, j\leq n}$ such that each $b_{i,j}=(\alpha_{i,j}^{(s)})_{1\leq s\leq d}\in \mathbb{R}^d$ is a null-sum vector for each $1\leq i, j\leq n$ (with the convention  that $b_{i,i}=0$ for all $1 \leq i\leq n$).  We define the nonlinear operator 
$$T({\bf A};{\bf b}):\overbrace{\mathbb{R}^d \times  \mathbb{R}^d \times \dots \times \mathbb{R}^d}^{n \text{ times}}\to \overbrace{\mathbb{R}^d \times \mathbb{R}^d \times \dots \times \mathbb{R}^d}^{n \text{ times}},$$ 
determined by 
$$T({\bf A};{\bf b})\begin{bmatrix}
v_1\\
v_2\\
\vdots\\
v_n
\end{bmatrix}=\begin{bmatrix}
{\displaystyle A_1v_1+\sum_{j=1}^n\langle v_1,v_j\rangle b_{1,j}}\\
{\displaystyle A_2v_2+\sum_{j=1}^n \langle v_2,v_j\rangle b_{2,j}}\\
\vdots\\
{\displaystyle A_nv_n+\sum_{j=1}^n \langle v_n,v_j\rangle b_{n,j}}
\end{bmatrix}.$$
The following are equivalent. 
\begin{enumerate}
\item For all the possible combinations of $\pm$ sign of the null-sum vectors $b_{i,j}$ we have that $T({\bf A};{\bf \pm b})$ induces a well-defined operator from $\Delta_{d-1}^n$ to  $\Delta_{d-1}^n$. 
\item For each $1\leq i\leq n$ and $1\leq s\leq d$ we have $${\displaystyle \sum_{j=1}^n|\alpha_{i,j}^{(s)}|\leq \min_{1\leq t\leq d} \{a_{i}^{(s,t)}\}}.$$  
\end{enumerate} 
\end{theorem}
\begin{proof} The proof is similar to the one for Proposition \ref{prop1}. First we want to show that the above condition is sufficient. Let $v_i=\begin{bmatrix}
v_i^{(1)}\\
\vdots\\
v_i^{(d)}
\end{bmatrix}\in \Delta_{d-1}$ for each $1\leq i\leq n$, and let 
$$w_i=\begin{bmatrix}
w_i^{(1)}\\
\vdots\\
w_i^{(d)}
\end{bmatrix}=A_iv_i+\sum_{j=1}^n\langle v_i,v_j\rangle b_{i,j}.$$ From Remark \ref{remark0} we know that $A_iv_i$ is a probability vector (i.e. the sum of its entries is $1$). Moreover, since $b_{i,j}$ are null-sum vectors we get that the sum of the entries of the vector $w_i$ is equal to $1$. We need to show that all the entries of $w_i$ are positive. 

Indeed, for each $1\leq s\leq d$ we have  
\begin{eqnarray*}
w_i^{(s)}&=&\sum_{t=1}^da_i^{(s,t)}v_i^{(t)}+\sum_{j=1}^n\langle v_i,v_j\rangle \alpha_{i,j}^{(s)}\\
&\geq & \min_{1\leq t\leq d} \{a_{i}^{(s,t)}\}-\sum_{j=1}^n|\alpha_{i,j}^{(s)}|\geq 0. 
\end{eqnarray*}
Here we  use that $0\leq a_i^{(s,t)}\leq 1$ (because $A_i$ is stochastic), $0\leq v_i^{(t)}\leq 1$,  and $0\leq \langle v_i,v_j\rangle\leq 1$ (it follows from Remark \ref{remark1} because $v_i, v_j\in \Delta_{d-1}$).  

Conversely, suppose that the operator is well-defined for every independent choice of signs of the null-sum vectors $b_{i,j}$. More precisely, for each $1\leq i,j\leq n$, let $\varepsilon_{i,j}\in{-1,1}$, and suppose that

$$
T\bigl({\bf A};(\varepsilon_{i,j}b_{i,j})_{1\leq i,j\leq n}\bigr):
\Delta_{d-1}^n\longrightarrow\Delta_{d-1}^n
$$

is well-defined for every choice of $(\varepsilon_{i,j})$.

Fix $1\leq i\leq n$, $1\leq s\leq d$, and $1\leq t\leq d$, and take

$$
v_k=e_t
$$

for every $1\leq k\leq n$. Then

$$
w_i^{(s)}
=
a_i^{(s,t)}
+
\sum_{j=1}^n
\varepsilon_{i,j}\alpha_{i,j}^{(s)}
\geq 0.
$$

Since the signs can be chosen independently, choose

$$
\varepsilon_{i,j}
=
-\operatorname{sgn}\bigl(\alpha_{i,j}^{(s)}\bigr)
$$

whenever $\alpha_{i,j}^{(s)}\neq 0$, with either choice when
$\alpha_{i,j}^{(s)}=0$. It follows that

$$
a_i^{(s,t)}
-
\sum_{j=1}^n
\left|\alpha_{i,j}^{(s)}\right|
\geq 0.
$$

Hence

$$
\sum_{j=1}^n
\left|\alpha_{i,j}^{(s)}\right|
\leq
a_i^{(s,t)}
$$

for every $1\leq t\leq d$. Taking the minimum over $t$ gives

$$
\sum_{j=1}^n
\left|\alpha_{i,j}^{(s)}\right|
\leq
\min_{1\leq t\leq d}
\left\{a_i^{(s,t)}\right\},
$$

which proves the necessity of condition $(2)$.

\end{proof}

\begin{corollary}  If the conditions of Theorem \ref{prop2} hold then the operator $T({\bf A};{\bf b})$ has a fixed point in $\Delta_{d-1}^n$. 
\end{corollary}
\begin{proof}
Since $\Delta_{d-1}^n$ is a compact and convex set the result follows from  Theorem \ref{FixedPoint} and Theorem \ref{prop2}.  
\end{proof}

\begin{remark} 
As in the case $n=2$ and $d=2$, if the conditions of Theorem \ref{prop2} are
satisfied, then Equation (\ref{seq3}) defines a well-defined stochastic process
that has a fixed point.
\end{remark}

\begin{remark} Unlike in the case $d=2$ and $n=2$ (see Example \ref{example2}), in the general case it is not clear if one has an algorithm for finding the exact values of the fixed points for $T({\bf A};{\pm \bf b})$. However, for practical purposes, one can use MATLAB simulations to approximate those fixed points. 
\end{remark}

Next, we will see a few examples.

\begin{example} \label{example4} Three divisions of the Umbrella Corporation have different cultures regarding health habits, namely Health-Conscious Culture (HC), Balanced-Flexible Culture (BF), and Indulgent Culture (IC). They all attend a joint 14-day retreat with three restaurants Green Bowl (healthy), Corner Bistro (balanced) and Big Burger House (unhealthy). 

Based on data from previous years (when the retreat was held separately), we know that  the day-to-day evolution of the restaurant of choice is given by Markov processes governed by $3\times 3$ stochastic matrices:

$$HC=\begin{pmatrix}
0.6 & 0.4 & 0.5\\
0.3 & 0.2 & 0.3\\
0.1 & 0.4 & 0.2
\end{pmatrix}, \; \; BF=\begin{pmatrix}
0.4 & 0.2 & 0.4\\
0.5 & 0.6 & 0.2\\
0.1 & 0.2 & 0.4
\end{pmatrix}, \; \; 
IC=\begin{pmatrix}
0.1 & 0.3 & 0.1\\
0.2 & 0.2 & 0.1\\
0.7 & 0.5 & 0.8
\end{pmatrix}.$$ 
The first column of HC describes where individuals who ate at Green Bowl
on the previous day eat on the current day. For example, the $(1,1)$ entry
corresponds to individuals who ate at Green Bowl on both days, the $(2,1)$
entry corresponds to those who moved from Green Bowl to Corner Bistro, and
the $(3,1)$ entry corresponds to those who moved from Green Bowl to Big
Burger House. Using MATLAB, we can approximate the fixed point  of this system to get 
$$v_{HC}\approx\begin{pmatrix}0.5253\\0.2727\\ 0.2020\end{pmatrix}, \; v_{BF}\approx\begin{pmatrix}0.3030\\ 0.4848\\ 0.212\end{pmatrix}, \; {\rm and} \; v_{IC}\approx\begin{pmatrix}0.1250\\ 0.1250\\ 0.7500\end{pmatrix}.$$ Alternatively, one can compute the  eigenvectors corresponding to $\lambda=1$.

Naturally, when attending the joint retreat,  members from different groups become curious and get influenced by the eating habits of  other groups. We will assume that these changes are modeled by Equation (\ref{seq3}) with the perturbation vectors:

$$b_{1,2}=b_{2,1}=\begin{pmatrix}
-0.05\\
0.01\\
0.04
\end{pmatrix}, \; \; b_{1,3}=b_{3,1}=\begin{pmatrix}
-0.03\\
0.05\\
-0.02
\end{pmatrix},\; \;  b_{2,3}=b_{3,2}=\begin{pmatrix}
-0.03\\
-0.01\\
0.04
\end{pmatrix}.$$

As all of the corresponding absolute values in our perturbation vectors $b_{i,j}$ and $b_{i,k}$ do not sum to any more than $0.08$, and the entries of the stochastic  matrices are at least $0.1$, the conditions of Theorem \ref{prop2} are satisfied.  
Therefore, there exists a solution  $(v_1, v_2, v_3)\in (\Delta_2)^3$ that satisfies Equation (\ref{seq3}). 

Using MATLAB, one can see that the system with interference has a fixed point given by 
$$v_1\approx\begin{pmatrix} 0.4968 \\ 0.2874 \\  0.2158 \end{pmatrix}, \; \; v_2\approx\begin{pmatrix} 0.2803 \\ 0.4743\\ 0.2454 \end{pmatrix}, \; {\rm and} \; v_3\approx\begin{pmatrix}0.1108\\ 0.1346\\  0.7546\end{pmatrix}.$$

\end{example}

\begin{example} \label{example5} The isolated northern region of Eldervale is home to four cities clustered within a fifty-mile radius: Ashgard, Brimhold, Colden, and Dunrock.
Each city has two types of workers: skilled workers and unskilled workers.  Until recently, the month-to-month evolution of the two types of workers in each of the four cities was given by Markov processes governed by $2\times2$ stochastic matrices

$$A=\begin{pmatrix}
0.6 & 0.7\\
0.4 & 0.3
\end{pmatrix}, 
\; B=\begin{pmatrix}
0.2 & 0.7 \\
0.8 & 0.3
\end{pmatrix}, \; 
C=\begin{pmatrix}
0.5 & 0.6 \\
0.5 & 0.4
\end{pmatrix},\;  {\rm and } \; 
D=\begin{pmatrix}
0.2 & 0.6\\
0.8 & 0.4
\end{pmatrix}.$$%
Using MATLAB, we can approximate the fixed point of this system to get   $$v_A\approx\begin{pmatrix} 0.63636\\
 0.36364\end{pmatrix},\; v_B\approx\begin{pmatrix} 0.46667\\ 0.53333\end{pmatrix},\; v_C\approx\begin{pmatrix} 0.54545\\ 0.45455\end{pmatrix},\; {\rm and }\;  v_D\approx\begin{pmatrix} 0.42857\\ 0.57143\end{pmatrix}.$$

A major infrastructure development transformed the region, creating the conditions for a mobile workforce able to commute daily between the cities. We will assume that these changes shifted the job market so that the workforce behaves according to the model given by Equation (\ref{seq3}) with perturbation vectors:

$$b_{1,2}=\begin{pmatrix}
0.05\\
-0.05
\end{pmatrix}, \; b_{1,3}=\begin{pmatrix}
-0.04\\
0.04
\end{pmatrix}, \;b_{1,4}=\begin{pmatrix}
0.07\\
-0.07
\end{pmatrix},$$ 
$$b_{2,1}=\begin{pmatrix}
0.03\\
-0.03
\end{pmatrix},\;
b_{2,3}=\begin{pmatrix}
0.01\\
-0.01
\end{pmatrix},\; b_{2,4}=\begin{pmatrix}
-0.02\\
0.02
\end{pmatrix},$$
$$b_{3,1}=\begin{pmatrix}
0.06\\
-0.06
\end{pmatrix},\; b_{3,2}=\begin{pmatrix}
-0.08\\
0.08
\end{pmatrix}, \; b_{3,4}=\begin{pmatrix}
0.09\\
-0.09
\end{pmatrix},$$
$$b_{4,1}=\begin{pmatrix}
0.04\\
-0.04
\end{pmatrix}, \; b_{4,2}=\begin{pmatrix}
0.05\\
-0.05
\end{pmatrix},\; b_{4,3}=\begin{pmatrix}
-0.01\\
0.01
\end{pmatrix}.$$

It is easy to see that the inequalities
\[
\sum_{j=1}^{4} |\alpha_{i,j}^{(s)}|
\leq
\min_{1\leq t\leq 2}\{a_i^{(s,t)}\}
\]
are satisfied for each $1\leq i\leq4$ and $1\leq s\leq2$.
Thus, by Theorem \ref{prop2}, there exists
\[
(v_1,v_2,v_3,v_4)\in(\Delta_1)^4
\]
satisfying Equation (\ref{seq3}).

Using MATLAB, one can see that the system with interference has a fixed point given by  $$v_1\approx\begin{pmatrix} 0.6704\\ 0.3296\end{pmatrix},\; v_2\approx\begin{pmatrix} 0.47309\\ 0.52691\end{pmatrix},\; v_3\approx\begin{pmatrix} 0.57848\\ 0.42152\end{pmatrix},\; {\rm and }\; v_4\approx\begin{pmatrix} 0.45685\\ 0.54315\end{pmatrix}.$$
\end{example}

\begin{remark}
One should note that in general $T({\bf A};{\pm \bf b})$ does not have a unique fixed point (consider for example the case when $A_i=I_d$ and $b_{i,j}=0$). In Section \ref{Section5} we give conditions for $T({\bf A};{\pm \bf b}):\Delta_{d-1}^n\to \Delta_{d-1}^n$ to be a contraction which imply uniqueness of the fixed point.
\end{remark}

\section{Contraction and Global Convergence}

\label{Section5}

In this section we investigate conditions under which the nonlinear operator
$T(A;b)$ is contractive. In addition to guaranteeing uniqueness of the
fixed point, contraction provides global convergence of the iterates and
an explicit geometric rate.

First, we introduce some notation. For $x=  \begin{pmatrix}
\alpha_1\\
\alpha_2\\
\vdots\\
\alpha_d
\end{pmatrix}\in \mathbb{R}^d$ we denote $\|x\|_1=\sum_{i=1}^d|\alpha_i|.$ Let
\[
\mathcal{S}=(\Delta_{d-1})^n.
\]
We equip $\mathcal{S}$ with the block metric
\[
d_\infty(x,y)
=
\max_{1\le i\le n}\|x_i-y_i\|_1,
\]
where
\[
x=  \begin{pmatrix}
x_1\\
x_2\\
\vdots\\
x_n
\end{pmatrix} ,
\qquad
y=\begin{pmatrix}
y_1\\
y_2\\
\vdots\\
y_n
\end{pmatrix}\in \mathcal{S}=(\Delta_{d-1})^n.
\]
We denote 
\[
T_i(x)
=
A_i x_i+
\sum_{j\ne i}
\langle x_i,x_j\rangle b_{i,j},
\qquad i=1,\ldots,n.
\]
In particular we have 
$$T({\bf A};{\bf b})(x)=\begin{pmatrix}
T_1(x)\\
T_2(x)\\
\vdots\\
T_n(x)
\end{pmatrix}.$$

We first record an elementary inequality on the probability simplex.

\begin{lemma}
Let $p,q,r\in\Delta_{d-1}$. Then
\[
\left|\langle p-q,r\rangle\right|
\le
\frac12\|p-q\|_1.
\]
Consequently, for $p,q,r,s\in\Delta_{d-1}$,
\[
\left|
\langle p,r\rangle-\langle q,s\rangle
\right|
\le
\frac12\|p-q\|_1
+
\frac12\|r-s\|_1.
\]
\end{lemma}

\begin{proof}
Set $z=p-q$. Since $p$ and $q$ are probability vectors,
\[
\sum_{k=1}^d z_k=0.
\]
Let
\[
P=\{k:z_k>0\},
\qquad
N=\{k:z_k<0\}.
\]
Then
\[
\sum_{k\in P}z_k
=
-\sum_{k\in N}z_k
=
\frac12\|z\|_1.
\]
Since $0\le r_k\le 1$ for all $k$,
\[
z^\top r
\le
\sum_{k\in P}z_k
=
\frac12\|z\|_1,
\]
and similarly
\[
z^\top r
\ge
\sum_{k\in N}z_k
=
-\frac12\|z\|_1.
\]
Hence
\[
|z^\top r|
\le
\frac12\|z\|_1.
\]
For the second assertion, write
\[
\langle p,r\rangle-\langle q,s\rangle
=
\langle p-q,r\rangle
+
\langle q,r-s\rangle.
\]
Applying the first assertion twice gives
\[
\left|
\langle p,r\rangle-\langle q,s\rangle
\right|
\le
\frac12\|p-q\|_1
+
\frac12\|r-s\|_1.
\]
\end{proof}

Recall from \cite{gm} that for a column-stochastic matrix $A=(a^{(s,t)})_{1\le s,t\le d}$, one can 
define its Dobrushin contraction coefficient by
\[
\delta(A)
=
\frac12
\max_{1\le r,t\le d}
\sum_{s=1}^d
\left|a^{(s,r)}-a^{(s,t)}\right|.
\]
Equivalently,
\[
\delta(A)
=
1-
\min_{1\le r,t\le d}
\sum_{s=1}^d
\min\left\{
a^{(s,r)},a^{(s,t)}
\right\}.
\]
For any $p,q\in\Delta_{d-1}$,
\[
\|A(p-q)\|_1
\le
\delta(A)\|p-q\|_1.
\]

We now obtain a global contraction condition.

\begin{theorem}[Global contraction]
\label{thm:global_contraction}
Assume that
\[
T({\bf A};{\bf b}):\mathcal S\longrightarrow\mathcal S
\]
is well-defined. Define
\[
q_i
=
\delta(A_i)
+
\sum_{j\ne i}\|b_{i,j}\|_1,
\qquad i=1,\ldots,n,
\]
and
\[
q=\max_{1\le i\le n}q_i.
\]
If
\[
q<1,
\]
then $T({\bf A};{\bf b})$ is a contraction on $(\mathcal S,d_\infty)$. More
precisely,
\[
d_\infty(T({\bf A};{\bf b})(x),T({\bf A};{\bf b})(y))
\le
q\,d_\infty(x,y)
\]
for every $x,y\in\mathcal S$.

Consequently, $T({\bf A};{\bf b})$ has a unique fixed point
\[
x^\ast=(x_1^\ast,\ldots,x_n^\ast)\in\mathcal S.
\]
Furthermore, for every $x^{(0)}\in\mathcal S$, the sequence
\[
x^{(m+1)}=T({\bf A};{\bf b})(x^{(m)})
\]
satisfies
\[
d_\infty(x^{(m)},x^\ast)
\le
q^m d_\infty(x^{(0)},x^\ast),
\]
and hence
\[
x^{(m)}\longrightarrow x^\ast
\]
geometrically as $m\to\infty$.
\end{theorem}

\begin{proof}
For $x,y\in\mathcal S$ and $1\le i\le n$,
\[
\begin{aligned}
T_i(x)-T_i(y)
={}&
A_i(x_i-y_i)\\
&+
\sum_{j\ne i}
\left[
\langle x_i,x_j\rangle
-
\langle y_i,y_j\rangle
\right]b_{i,j}.
\end{aligned}
\]
Taking the $\ell_1$-norm gives
\[
\begin{aligned}
\|T_i(x)-T_i(y)\|_1
\le{}&
\|A_i(x_i-y_i)\|_1\\
&+
\sum_{j\ne i}
\left|
\langle x_i,x_j\rangle
-
\langle y_i,y_j\rangle
\right|
\|b_{i,j}\|_1.
\end{aligned}
\]
By the Dobrushin contraction inequality,
\[
\|A_i(x_i-y_i)\|_1
\le
\delta(A_i)\|x_i-y_i\|_1.
\]
By the preceding lemma,
\[
\left|
\langle x_i,x_j\rangle
-
\langle y_i,y_j\rangle
\right|
\le
\frac12\|x_i-y_i\|_1
+
\frac12\|x_j-y_j\|_1.
\]
Therefore,
\[
\begin{aligned}
\|T_i(x)-T_i(y)\|_1
\le{}&
\delta(A_i)\|x_i-y_i\|_1\\
&+
\frac12
\sum_{j\ne i}
\|b_{i,j}\|_1
\left(
\|x_i-y_i\|_1+\|x_j-y_j\|_1
\right).
\end{aligned}
\]
Since
\[
\|x_k-y_k\|_1
\le
d_\infty(x,y)
\]
for every $k$,
\[
\|T_i(x)-T_i(y)\|_1
\le
\left[
\delta(A_i)+
\sum_{j\ne i}\|b_{i,j}\|_1
\right]
d_\infty(x,y).
\]
Thus
\[
\|T_i(x)-T_i(y)\|_1
\le
q_i d_\infty(x,y).
\]
Taking the maximum over $i$ yields
\[
d_\infty(T({\bf A};{\bf b})(x),T({\bf A};{\bf b})(y))
\le
q\,d_\infty(x,y).
\]
If $q<1$, then $T({\bf A};{\bf b})$ is a contraction.

Since $\mathcal S$ is a closed subset of a finite-dimensional normed
space, $(\mathcal S,d_\infty)$ is complete. The Banach fixed-point
theorem therefore implies existence and uniqueness of a fixed point
$x^\ast$, together with
\[
d_\infty(T({\bf A};{\bf b})^m(x^{(0)}),x^\ast)
\le
q^m d_\infty(x^{(0)},x^\ast).
\]
\end{proof}

The contraction criterion has a direct connection with the
stochasticity conditions established in Theorem~4.1. For
$1\le i\le n$ and $1\le s\le d$, define
\[
m_i^{(s)}
=
\min_{1\le t\le d}
a_i^{(s,t)}
\]
and
\[
\kappa_i
=
\sum_{s=1}^d m_i^{(s)}.
\]

\begin{lemma}
\label{lem:dobrushin_bound}
For every $i=1,\ldots,n$,
\[
\delta(A_i)\le 1-\kappa_i.
\]
\end{lemma}

\begin{proof}
Using the alternative representation of the Dobrushin coefficient,
\[
1-\delta(A_i)
=
\min_{1\le r,t\le d}
\sum_{s=1}^d
\min\left\{
a_i^{(s,r)},a_i^{(s,t)}
\right\}.
\]
For every $r,t$ and $s$,
\[
m_i^{(s)}
=
\min_{1\le \ell\le d}
a_i^{(s,\ell)}
\le
\min\left\{
a_i^{(s,r)},a_i^{(s,t)}
\right\}.
\]
Therefore,
\[
\kappa_i
\le
\sum_{s=1}^d
\min\left\{
a_i^{(s,r)},a_i^{(s,t)}
\right\}
\]
for every $r,t$. Taking the minimum over $r,t$ gives
\[
\kappa_i
\le
1-\delta(A_i),
\]
which proves the result.
\end{proof}

\begin{corollary}[Nonexpansiveness under the stochasticity condition]
\label{cor:nonexpansive}
Suppose the conditions of Theorem~4.1 hold, that is,
\[
\sum_{j=1}^n
|\alpha_{i,j}^{(s)}|
\le
\min_{1\le t\le d}a_i^{(s,t)}
\]
for every $1\le i\le n$ and $1\le s\le d$. Then
\[
d_\infty(T({\bf A};{\bf b})(x),T({\bf A};{\bf b})(y))
\le
d_\infty(x,y)
\]
for all $x,y\in\mathcal S$. Thus $T({\bf A};{\bf b})$ is nonexpansive.
\end{corollary}

\begin{proof}
Since
\[
b_{i,j}
=
\left(
\alpha_{i,j}^{(1)},\ldots,
\alpha_{i,j}^{(d)}
\right)^\top,
\]
we have
\[
\begin{aligned}
\sum_{j\ne i}\|b_{i,j}\|_1
&=
\sum_{j\ne i}
\sum_{s=1}^d
|\alpha_{i,j}^{(s)}|\\
&=
\sum_{s=1}^d
\sum_{j\ne i}
|\alpha_{i,j}^{(s)}|\\
&\le
\sum_{s=1}^d m_i^{(s)}
=
\kappa_i.
\end{aligned}
\]
By Lemma~\ref{lem:dobrushin_bound},
\[
\delta(A_i)\le1-\kappa_i.
\]
Hence
\[
\delta(A_i)+
\sum_{j\ne i}\|b_{i,j}\|_1
\le1.
\]
The conclusion now follows from the proof of
Theorem~\ref{thm:global_contraction}.
\end{proof}

We can strengthen the preceding result whenever there is strict slack
in the stochasticity inequalities. Define
\[
\varepsilon_i
=
\sum_{s=1}^d
\left[
m_i^{(s)}
-
\sum_{j\ne i}
|\alpha_{i,j}^{(s)}|
\right].
\]

\begin{corollary}[Strict stochasticity and global convergence]
\label{cor:strict_contraction}
Suppose the assumptions of Theorem~4.1 hold and
\[
\varepsilon_i>0,
\qquad i=1,\ldots,n.
\]
Then $T({\bf A};{\bf b})$ is a contraction. In particular,
\[
q
\le
1-\min_{1\le i\le n}\varepsilon_i
<1.
\]
Consequently, $T({\bf A};{\bf b})$ has a unique fixed point in $\mathcal S$, and
every orbit converges geometrically to this fixed point.

In particular, it is sufficient that, for every $i$, at least one of
the inequalities
\[
\sum_{j\ne i}
|\alpha_{i,j}^{(s)}|
\le
\min_{1\le t\le d}
a_i^{(s,t)},
\qquad s=1,\ldots,d,
\]
is strict.
\end{corollary}

\begin{proof}
By definition,
\[
\sum_{j\ne i}\|b_{i,j}\|_1
=
\kappa_i-\varepsilon_i.
\]
Thus, using Lemma~\ref{lem:dobrushin_bound},
\[
\begin{aligned}
\delta(A_i)
+
\sum_{j\ne i}\|b_{i,j}\|_1
&\le
1-\kappa_i+\kappa_i-\varepsilon_i\\
&=
1-\varepsilon_i.
\end{aligned}
\]
Therefore
\[
q
\le
1-\min_i\varepsilon_i<1.
\]
The conclusion follows from
Theorem~\ref{thm:global_contraction}.
\end{proof}

\begin{remark} Note that for Examples \ref{example4} and \ref{example5} the inequality $\sum_{j\ne i}
|\alpha_{i,j}^{(s)}|
\le
\min_{1\le t\le d}
a_i^{(s,t)}$ is always strict in both examples; in particular, the fixed points are unique. 
\end{remark}

\section{Discussion and Further Directions}\label{Section6}

The model introduced in this paper separates baseline Markov dynamics from pairwise interactions through quadratic perturbation terms. This structure allows explicit conditions for stochasticity and, under stronger assumptions, uniqueness and global convergence of the fixed point.

The perturbation vectors also admit a natural interpretation. If

$$
b_{j,i}=-b_{i,j},
$$
the interaction between the $i$-th and $j$-th components is antagonistic, whereas
$$
b_{j,i}=b_{i,j}
$$
corresponds to a synergistic interaction. More generally, no symmetry between $b_{i,j}$ and $b_{j,i}$ is required, allowing asymmetric interactions as well.

The factor
$$
\langle X_i(m),X_j(m)\rangle
$$
controls the strength of the interaction. It lies in $[0,1]$ and measures overlap between the two probability vectors. However, it is not a normalized similarity measure. For example, if
$$
X_1=X_2=
\begin{pmatrix}
1/2\\
1/2
\end{pmatrix},
$$
then $\langle X_1,X_2\rangle=1/2$, even though the two distributions are identical. Thus, the inner product favors overlap concentrated near the standard basis vectors.

This suggests replacing the inner product by a more general similarity function
$$
K:\Delta_{d-1}\times\Delta_{d-1}\longrightarrow[0,1]
$$
and considering
$$
X_i(m+1)
=
A_iX_i(m)
+
\sum_{j\ne i}
K\bigl(X_i(m),X_j(m)\bigr)b_{i,j}.
$$

The sufficiency argument in Theorem~4.1 continues to hold whenever $0\leq K(x,y)\leq1$. If, in addition,
$$
K(e_t,e_t)=1,\qquad t=1,\ldots,d,
$$
the same vertex argument yields necessity. Extending the contraction results of Section~5 would require an appropriate Lipschitz condition on $K$.

Theorem~4.1 may also be viewed as a robust admissibility result: its conditions depend only on the magnitudes of the perturbations and guarantee stochasticity under every possible sign reversal of the vectors $b_{i,j}$. Section~5 further shows that these stochasticity conditions imply nonexpansiveness, while strict inequalities imply contraction and hence a unique globally attracting fixed point. The boundary case, where one or more inequalities hold with equality, remains an interesting direction for further study.

Finally, parameter estimation provides a natural statistical extension. If the matrices $A_i$ are known and the states are observed, then
$$
X_i(m+1)-A_iX_i(m)
=
\sum_{j\ne i}
\langle X_i(m),X_j(m)\rangle b_{i,j},
$$
which is linear in the unknown perturbation vectors. This suggests constrained estimation procedures incorporating the null-sum and stochasticity conditions. Joint estimation of the matrices $A_i$ and the perturbation vectors is another possible direction for future work.



\section*{Statements and Declarations}

The authors have no relevant financial or non-financial interests to disclose. Data sharing is not applicable to this article as no data sets were generated or analyzed during the current study.

The authors report that generative AI was not used in their research or preparation of this manuscript.

\section*{Acknowledgment}   We thank  Kit Chan and John Chen  for some discussions on an earlier version of the manuscript. This work originated as part of an undergraduate research project conducted during the summer of 2026. 

\bibliographystyle{amsalpha}

 \end{document}